\documentclass{article}
\usepackage[total={5.95in, 7.7in}, voffset=0in]{geometry} %
\usepackage{enumitem}
\usepackage{graphicx}
\usepackage{subfigure}
\usepackage{tikz, everypage}
\usetikzlibrary{patterns}

\usepackage[ruled,vlined]{algorithm2e}
\usepackage{amsmath,latexsym,amssymb,xcolor}
\usepackage{amsthm}
\usepackage{bbm}
\makeatother
\usepackage[pdfborder={0 0 0}]{hyperref}
\hypersetup{
  urlcolor = black,
  pdfauthor = {Alberto Gonzalez-Sanz and Marcel Nutz},
  pdfkeywords = {Linear Program, Optimal Transport, Minimal-norm Solution, Quadratic Regularization},
  pdftitle = {Quantitative convergence of quadratically regularized linear programs},
  pdfsubject = {Quantitative convergence of quadratically regularized linear programs},
  pdfpagemode = UseNone
}

\theoremstyle{plain}
\newtheorem {Proposition}{Proposition}[section]
\newtheorem {Lemma}[Proposition] {Lemma}
\newtheorem {Theorem}[Proposition]{Theorem}
\newtheorem {Corollary}[Proposition]{Corollary}
\theoremstyle{definition}
\newtheorem {Remark}[Proposition]{Remark}

\newtheorem {Example}{Example}[section]

\def\N{\mathbb{N}}

\def\R{\mathbb{R}}

\def\x{\mathbf{x}}
\def\v{\mathbf{v}}
\def\c{\mathbf{c}}

\def\X{\mathbf{X}}
\def\Y{\mathbf{Y}}

\def\cM{\mathcal{M}}
\def\cP{\mathcal{P}}
\def\cF{\mathcal{F}}

\DeclareMathOperator{\ep}{exp}
\DeclareMathOperator{\proj}{proj}

\DeclareMathOperator{\id}{Id}

\DeclareMathOperator*{\argmin}{arg\, min}

\renewenvironment{thebibliography}[1]{%
\begin{oldthebibliography}{#1}%
\setlength{\baselineskip}{.9em}
\linespread{.95}
\small
\setlength{\parskip}{.23ex}%
\setlength{\itemsep}{.15em}%
}%
{%
\end{oldthebibliography}%
}
\begin{document}

\title{\vspace{-1em}
Quantitative Convergence of Quadratically Regularized Linear Programs\footnote{The authors thank Roberto Cominetti, Andr{\'e}s Riveros Valdevenito and two anonymous referees for helpful comments.}}
\author{
  Alberto Gonz{\'a}lez-Sanz%
  \thanks{
  Columbia University, Department of Statistics, ag4855@columbia.edu.}
  \and
  Marcel Nutz%
  \thanks{
  Columbia University, Departments of Statistics and Mathematics, mnutz@columbia.edu. Research supported by NSF Grants DMS-1812661, DMS-2106056, DMS-2407074.}
  }
\date{\today}

\maketitle

\begin{abstract}
Linear programs with quadratic (``ridge'') regularization are of recent interest in optimal transport: unlike entropic regularization, the squared-norm penalty gives rise to sparse approximations of optimal transport couplings. More broadly, quadratic regularization is used in overparametrized learning problems to single out a particular solution. It is well known that the solution of a quadratically regularized linear program over any polytope converges stationarily to the minimal-norm solution of the linear program when the regularization parameter tends to zero. However, that result is merely qualitative. Our main result quantifies the convergence by specifying the exact threshold for the regularization parameter, after which the regularized solution also solves the linear program. Moreover, we bound the suboptimality of the regularized solution before the threshold. These results are complemented by a convergence rate for the regime of large regularization. We apply our general results to the setting of optimal transport, where we shed light on how the threshold and suboptimality depend on the number of data points.
\end{abstract}

\vspace{1.5em}

{\small

}
\vspace{.9em}

\section{Introduction} 

Let $\c\in\R^{d}$ and let $\mathcal{P}\subset \R^d$ be a polytope. Moreover, let $\langle \cdot , \cdot \rangle $ be an inner product on $\R^d$ and $\|\cdot\|$ its induced norm. We study the linear program 
\begin{align}{\tag{LP}}\label{LP}
\begin{split}
    \text{minimize}~~ \langle \c, \x \rangle \qquad
    \text{subject to}~~\x\in \mathcal{P}
\end{split}
\end{align}
and its quadratically regularized counterpart,
\begin{align}{\tag{QLP}}\label{QLP}
\begin{split}
    \text{minimize}~~ \langle \c, \x \rangle +\frac{\|\x\|^2}{\eta} \qquad
    \text{subject to}~~\x\in \mathcal{P}.
\end{split}
\end{align}
Here $\eta \in (0,\infty)$ is called the inverse regularization parameter (whereas $1/\eta$ is the regularization). In the limit $\eta \to\infty$ of small regularization, \eqref{QLP} converges to \eqref{LP}. More precisely, the unique solution $\x^{\eta}$ of \eqref{QLP} converges to a particular solution $\x^{*}$ of \eqref{LP}, namely the solution with smallest norm: 
$\x^* =\arg\min_{\x\in  \cM } \|\x\|^2 $, where $\cM $ denotes the set of minimizers of \eqref{LP}. Our main goal is to describe how quickly this convergence happens.

Linear programming and regularization are fundamental tools in data science. Many statistical methodologies, including for instance quantile regression \cite{Koenker1978},  statistical depths \cite{Hallin2010} or multivariate quantiles \cite{HallinDelBarrioCuestaAlbertosMatran.21}, rely on solving a linear program. Regularization by a quadratic penalty---also called ridge penalty due to its prominent application in ridge regression---is used in many statistical problems (e.g., regularized quantile regression \cite{LiBayesian2010}) but also in data science more broadly, for instance in overparametrized learning problems where the aim is to single out a particular solution~\cite{BartlettLongLugosiTsigler.20, Belkin.21, Zhou2024}.

The aforementioned convergence of the solution $\x^{\eta}$ of \eqref{QLP} to the minimum-norm solution~$\x^{*}$ of \eqref{LP} is stationary: there exists a threshold $\eta^{*}$ such that $\x^{\eta}=\x^{*}$ for all $\eta\geq \eta^{*}$. This was first established for linear programs in \cite[Theorem~1]{MangasarianMeyer.79} and \cite[Theorem~2.1]{Mangasarian.84}, and was more recently rediscovered in the context of optimal transport \cite[Property~5]{DesseinPapadakisRouas.18}. However, those results are qualitative: they do not give a value or a bound for~$\eta^{*}$. We shall characterize the exact value of the threshold~$\eta^{*}$ (cf.\ Theorem~\ref{th:main}), and show how this leads to computable bounds in applications. This exact result raises the question about the speed of convergence as $\eta\uparrow \eta^{*}$. Specifically, we are interested in the convergence of the error $
\mathcal{E}(\eta) = \langle \mathbf{c}, \x^{\eta} \rangle - \min_{\x \in \mathcal{P}} \langle \mathbf{c}, \x \rangle
$ measuring how suboptimal the solution $\x^\eta$ of \eqref{QLP} is when plugged into  \eqref{LP}. In Theorem~\ref{th:main}, we show that $
\mathcal{E}(\eta) = o(\eta^{*}-\eta)$ as $\eta\uparrow \eta^{*}$ and give an explicit bound for $\mathcal{E}(\eta)/(\eta^{*}-\eta)$. After observing that the curve $\eta\mapsto \x^{\eta}$ is piecewise affine, this linear rate can be understood as the slope of the last segment of the curve before ending at $\x^{*}$. Figure~\ref{fig:Motivatingexample} illustrates these quantities in a simple example. Our results for $\eta\to\infty$ are complemented by a convergence rate for the large regularization regime $\eta\to0$ where~$\x^{\eta}$ tends to $\arg\min_{\x\in  \mathcal{P}} \|\x\|^2 $; cf.\ Proposition~\ref{pr:largeRegularization}.

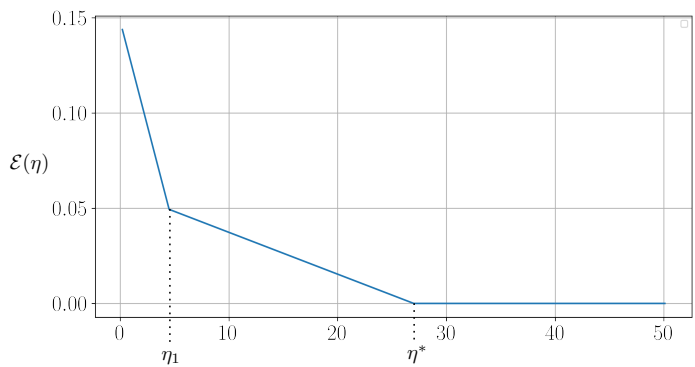
\begin{figure}[h!] 
		\centering
		\resizebox{.7\linewidth}{!}{%

\tikzset{every picture/.style={line width=0.75pt}} %

\begin{tikzpicture}[x=0.75pt,y=0.75pt,yscale=-1,xscale=1]
\draw (265.5,149) node  {\includegraphics[width=372.75pt,height=193.5pt]{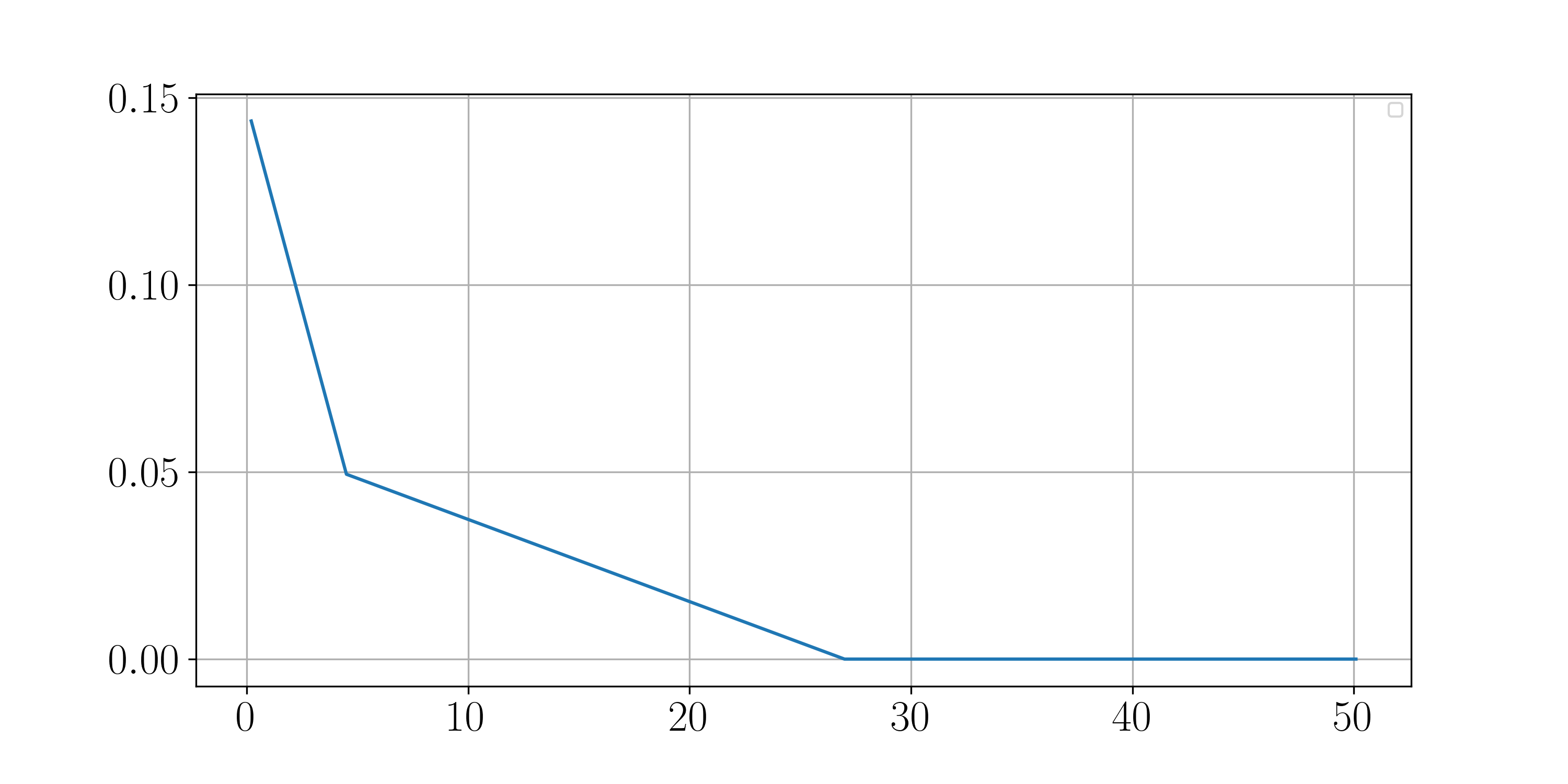}};
\draw  [dash pattern={on 0.84pt off 2.51pt}]  (127,176) -- (127,262) ;
\draw  [dash pattern={on 0.84pt off 2.51pt}]  (285,237) -- (285,263) ;

\draw (279,262.4) node [anchor=north west][inner sep=0.75pt]    {$\eta^{*}$};
\draw (22,138.4) node [anchor=north west][inner sep=0.75pt]    {$\mathcal{E}(\eta)$};
\draw (120,266.4) node [anchor=north west][inner sep=0.75pt]    {$\eta_{1}$};

\end{tikzpicture}
}%

\caption{Suboptimality $\mathcal{E}(\eta)$ of \eqref{QOT} when $\mu = \nu =\frac{1}{3} \sum_{i=1}^{3} \delta_{i/3}$ and $c(x,y) = \|x-y\|^2$. Theorem~\ref{th:main} characterizes the location of $\eta^{*}$ and bounds the slope to the left of $\eta^{*}$.}
\label{fig:Motivatingexample}
\end{figure}

While linear programs and their penalized counterparts go back far into the last century,
our interest 
is fueled by the surge of optimal transport in applications such as machine learning (e.g., \cite{KolouriParlEtAl.17survey}), statistics (e.g., \cite{PanaretosZemel.19}), language and image processing (e.g., \cite{AlvarezJaakkola.18, RubnerTomasiGuibas.00}) and economics (e.g., \cite{Galichon.16}). 
In its simplest form, the optimal transport problem between probability measures~$\mu$ and~$\nu$ is
\begin{equation}\tag{OT}\label{OTCONt}
\inf_{\gamma\in \Gamma(\mu, \nu)} \int c(x,y) d\gamma(x,y), 
 \end{equation}
where $\Gamma(\mu, \nu)$ denotes the set of couplings; i.e., probability measures $\gamma$ with marginals  $\mu$ and $\nu$ (see \cite{Villani.03,Villani.09}  for an in-depth exposition). Here $c(\cdot,\cdot)$ is a given cost function, most commonly $c(x,y)=\|x-y\|^2$. In many applications the marginals represent observed data: data points ${\X_1, \dots, \X_N}$ and ${\Y_1, \dots, \Y_N}$ are encoded in their empirical measures $\mu=\frac1N \sum_{i} \delta_{\X_i}$ and $\nu=\frac1N \sum_{i} \delta_{\Y_i}$. Writing also $\c_{ij}=c(\X_{i},\Y_{j})$, the problem \eqref{OTCONt} is a particular case of \eqref{LP} in dimension $d=N\times N$. The general linear program \eqref{LP} also includes other transport problems of recent interest, such as multi-marginal optimal transport and Wasserstein barycenters~\cite{AguehCarlier.11}, adapted Wasserstein distances~\cite{BackhoffBartlBeiglbockEder.20} or martingale optimal transport~\cite{BeiglbockHenryLaborderePenkner.11}. 

As the optimal transport problem is computationally costly (e.g., \cite{PeyreCuturi.19}), \cite{Cuturi.13}  proposed to regularize~\eqref{OTCONt} by penalizing with Kullback--Leibler divergence (entropy). Then, solutions can be computed using the Sinkhorn--Knopp (or IPFP) algorithm, which has lead to an explosion of high-dimensional applications. 
Entropic regularization always leads to ``dense'' solutions (couplings whose support contains all data pairs $(\X_i,\Y_j)$) even though the unregularized problem~\eqref{OTCONt}  typically has a sparse solution. In some applications that is undesirable; for instance, it may correspond to blurrier images in an image processing task~\cite{blondel18quadratic}. For that reason, \cite{blondel18quadratic} suggested the quadratic penalization
\begin{equation}
    {\tag{QOT}}\label{QOT}
    \inf_{\gamma\in \Gamma(\mu, \nu)} \int c(x,y)d\gamma(x,y)+\frac{1}{\eta}  \left\| \frac{d\gamma}{  d(\mu \otimes \nu) }\right\|_{L^2( \mu \otimes \nu)}^2
\end{equation} 
where $ d\gamma/  d(\mu \otimes \nu) $ denotes the density of $\gamma$ with respect to the  product measure $\mu \otimes \nu$. See also \cite{EssidSolomon.18} for  a similar formulation of minimum-cost flow problems, the predecessors referenced therein, and \cite{DesseinPapadakisRouas.18} for optimal transport with more general convex regularization. Quadratic regularization gives rise to sparse solutions (see \cite{blondel18quadratic}, and \cite{GonzalezSanzNutz.24b, GonzalezSanzNutzRiveros.24, Nutz.24, WieselXu.24} for recent theoretical results). Applications of quadratically regularized optimal transport include manifold learning~\cite{ZhangMordantMatsumotoSchiebinger.23} and image processing \cite{LiGenevayYurochkinSolomon.20} while \cite{Mordant.23} establishes a connection to maximum likelihood estimation of Gaussian mixtures. Computational approaches are developed in \cite{EcksteinKupper.21, GeneveyEtAl.16, GulrajaniAhmedArjovskyDumoulinCourville.17, LiGenevayYurochkinSolomon.20, seguy2018large} whereas \cite{LorenzMannsMeyer.21, DiMarinoGerolin.20b, BayraktarEcksteinZhang.22, Nutz.24} study theoretical aspects with a focus on continuous problems. In that context, \cite{LorenzMahler.22, EcksteinNutz.22} show Gamma convergence to the unregularized optimal transport problem in the small regularization limit. Those results are straightforward in the discrete case considered in the present work. Conversely, the stationary convergence studied here does not take place in the continuous case.

For linear programs with entropic regularization, \cite{CominettiSanMartin.94} established that solutions converge exponentially to the limiting unregularized counterpart.  More recently, \cite{Weed.18} gave an explicit bound for the convergence rate. The picture for entropic regularization is quite different to quadratic regularization as the convergence is not stationary. For instance, in optimal transport, the support of the regularized solution contains all data pairs for any value of the regularization parameter, collapsing only at the unregularized limit. Nevertheless, our analysis benefits from some of the technical ideas in \cite{Weed.18}, specifically for the proof of the slope bound~\eqref{eq:SlopeBoundInThm}. The small regularization limit has also attracted a lot of attention in continuous optimal transport (e.g., \cite{AltschulerNilesWeedStromme.21,BerntonGhosalNutz.21, CarlierDuvalPeyreSchmitzer.17, ConfortiTamanini.19,Leonard.12, NutzWiesel.21, Pal.19}) which however is technically less related to the present work.

The remainder of this note is organized as follows. Section~\ref{se:main} contains the main results on the general linear program and its quadratic regularization, Section~\ref{se:OT} the application to optimal transport. Proofs are gathered in Section~\ref{se:proofs}.

\section{Main Results}\label{se:main}

Throughout, $\emptyset\neq\mathcal{P}\subset\R^{d}$ denotes a polytope. That is, $\mathcal{P}$ is the convex hull of its extreme points (or vertices) $\ep(\cP)=\{ \v_1, \dots, \v_K\}$, which are in turn minimal with the property of spanning~$\cP$ (see \cite{Brondsted.83} for detailed definitions). 
We recall the linear program \eqref{LP} and its quadratically penalized version \eqref{QLP} as defined in the Introduction, and in particular their cost vector $\c\in\R^{d}$. The set of minimizers of~\eqref{LP} is denoted
$$ \cM =\cM (\mathcal{P}, \c)=\argmin_{\x\in \mathcal{P}}\langle \c, \x \rangle;$$
it is again a polytope. To avoid a degenerate problem, we assume throughout that the projection of the origin onto $\mathcal{P}$ is not a minimizer of~\eqref{LP}. (If it is a minimizer of~\eqref{LP}, then it is also the minimizer of \eqref{QLP} for any $\eta$, so that our problem is trivial.) We abbreviate the objective function of \eqref{QLP} as
$$
    \Phi_\eta (\x) =  \langle \c, \x \rangle +\frac{\|\x\|^2}{\eta}.
$$
In view of $\Phi_\eta (\x)= \frac{1}{\eta} \left\|  \x + \frac{ \eta \, \mathbf{c} }{2} \right\|^2 - \frac{\eta}{4}\|\mathbf{c}\|^2$, minimizing $\Phi_\eta (\x)$ over $\mathcal{P}$ is equivalent to projecting $-\eta \c/2$ onto $\mathcal{P}$ in the Hilbert space $(\R^{d}, \langle \cdot , \cdot \rangle )$. The projection theorem (e.g., \cite[Theorem~5.2]{Brezis.11}) thus implies the following result. We denote by \({\rm ri}(C)\)  the relative interior of a set \(C \subset \mathbb{R}^d\); i.e, the topological interior when \(C\) is considered as a subset of its affine hull.

\begin{Lemma}\label{Lemma:projection}
  Given $\eta>0$, \eqref{QLP} admits a unique minimizer $\x^{\eta}$. 
  It is characterized as the unique $\x^\eta\in \mathcal{P}$ such that
    $$
\left\langle -\frac{\eta \c}{2}-\x^\eta,  \x-\x^\eta \right\rangle \leq 0   \quad \text{for all } \x\in \mathcal{P} .
$$
In particular, if $\x^\eta\in {\rm ri}(C)$ for some convex set $C\subset\mathcal{P}$, then also
    $$
\left\langle -\frac{\eta \c}{2}-\x^\eta,  \x-\x^\eta \right\rangle  =0  \quad \text{for all } \x\in C .
$$
\end{Lemma}

Figure~\ref{fig:projection} illustrates how $\x^{\eta}$ is obtained as the projection of $-\eta \c/2$. %
The algorithm of~\cite{HagerZhang.16} solves the problem of projecting a point onto a polyhedron, hence can be used to find~$\x^{\eta}$ numerically.

\begin{figure}[tbh] 
\centering
		\resizebox{.5\linewidth}{!}{%
		\tikzset{every picture/.style={line width=0.75pt}} %
		\begin{tikzpicture}[x=0.75pt,y=0.75pt,yscale=-1,xscale=1]
\draw [line width=1.5]    (212,240) -- (414.17,37.83) ;
\draw [shift={(417,35)}, rotate = 135] [fill={rgb, 255:red, 0; green, 0; blue, 0 }  ][line width=0.08]  [draw opacity=0] (11.61,-5.58) -- (0,0) -- (11.61,5.58) -- cycle    ;
\draw  [line width=1.5]  (282,211) -- (341,211) -- (341,270) -- (282,270) -- cycle ;
\draw  [dash pattern={on 0.84pt off 2.51pt}]  (212,240) -- (280,240) ;
\draw [shift={(282,240)}, rotate = 180] [color={rgb, 255:red, 0; green, 0; blue, 0 }  ][line width=0.75]    (10.93,-3.29) .. controls (6.95,-1.4) and (3.31,-0.3) .. (0,0) .. controls (3.31,0.3) and (6.95,1.4) .. (10.93,3.29)   ;
\draw  [dash pattern={on 0.84pt off 2.51pt}]  (241,211) -- (279,211) ;
\draw [shift={(281,211)}, rotate = 180] [color={rgb, 255:red, 0; green, 0; blue, 0 }  ][line width=0.75]    (10.93,-3.29) .. controls (6.95,-1.4) and (3.31,-0.3) .. (0,0) .. controls (3.31,0.3) and (6.95,1.4) .. (10.93,3.29)   ;
\draw  [dash pattern={on 0.84pt off 2.51pt}]  (282,170) -- (282,209) ;
\draw [shift={(282,211)}, rotate = 270] [color={rgb, 255:red, 0; green, 0; blue, 0 }  ][line width=0.75]    (10.93,-3.29) .. controls (6.95,-1.4) and (3.31,-0.3) .. (0,0) .. controls (3.31,0.3) and (6.95,1.4) .. (10.93,3.29)   ;
\draw  [dash pattern={on 0.84pt off 2.51pt}]  (311.5,141) -- (311.5,208.5) ;
\draw [shift={(311.5,210.5)}, rotate = 270] [color={rgb, 255:red, 0; green, 0; blue, 0 }  ][line width=0.75]    (10.93,-3.29) .. controls (6.95,-1.4) and (3.31,-0.3) .. (0,0) .. controls (3.31,0.3) and (6.95,1.4) .. (10.93,3.29)   ;
\draw  [dash pattern={on 0.84pt off 2.51pt}]  (341,112) -- (341,209) ;
\draw [shift={(341,211)}, rotate = 270] [color={rgb, 255:red, 0; green, 0; blue, 0 }  ][line width=0.75]    (10.93,-3.29) .. controls (6.95,-1.4) and (3.31,-0.3) .. (0,0) .. controls (3.31,0.3) and (6.95,1.4) .. (10.93,3.29)   ;
\draw  [dash pattern={on 0.84pt off 2.51pt}]  (371,81) -- (341.45,209.05) ;
\draw [shift={(341,211)}, rotate = 282.99] [color={rgb, 255:red, 0; green, 0; blue, 0 }  ][line width=0.75]    (10.93,-3.29) .. controls (6.95,-1.4) and (3.31,-0.3) .. (0,0) .. controls (3.31,0.3) and (6.95,1.4) .. (10.93,3.29)   ;
\draw  [dash pattern={on 0.84pt off 2.51pt}]  (261,190) -- (280.59,209.59) ;
\draw [shift={(282,211)}, rotate = 225] [color={rgb, 255:red, 0; green, 0; blue, 0 }  ][line width=0.75]    (10.93,-3.29) .. controls (6.95,-1.4) and (3.31,-0.3) .. (0,0) .. controls (3.31,0.3) and (6.95,1.4) .. (10.93,3.29)   ;
\draw  [dash pattern={on 0.84pt off 2.51pt}]  (404,48) -- (341.72,209.13) ;
\draw [shift={(341,211)}, rotate = 291.13] [color={rgb, 255:red, 0; green, 0; blue, 0 }  ][line width=0.75]    (10.93,-3.29) .. controls (6.95,-1.4) and (3.31,-0.3) .. (0,0) .. controls (3.31,0.3) and (6.95,1.4) .. (10.93,3.29)   ;

\draw (171,233.4) node [anchor=north west][inner sep=0.75pt]    {$\eta =0$};
\draw (221,176.4) node [anchor=north west][inner sep=0.75pt]    {$\eta =1$};
\draw (271,128.4) node [anchor=north west][inner sep=0.75pt]    {$\eta =2$};
\draw (301,94.4) node [anchor=north west][inner sep=0.75pt]    {$\eta =\eta ^{*}$};
\draw (419,45.4) node [anchor=north west][inner sep=0.75pt]    {$( -\eta \c/2)_{\eta \geq 0}{}$};
\draw (305,233.4) node [anchor=north west][inner sep=0.75pt]    {$\mathcal{P}$};
\draw (348,203.4) node [anchor=north west][inner sep=0.75pt]    {$\x^{*}$};

\end{tikzpicture}
}%

\caption{The minimizer $\x^{\eta}$ of \eqref{QLP} is the projection of $-\eta \c/2$ onto~$\mathcal{P}$. The curve $\eta\mapsto \x^{\eta}$ is piecewise affine and converges stationarily to a point $\x^{*}$; i.e., $\x^{\eta}=\x^{*}$ for all $\eta\geq \eta^{*}$.}
\label{fig:projection}
\end{figure}
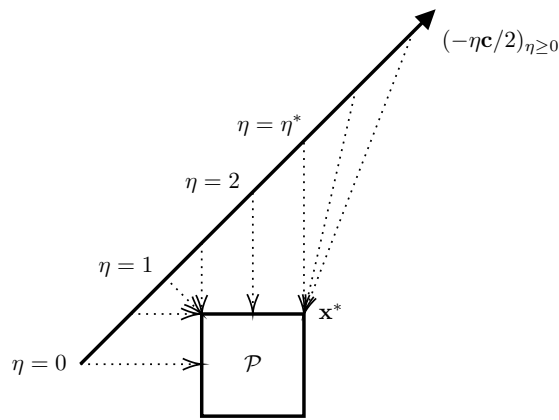

Next, we are interested in the error or \emph{suboptimality}
\begin{equation}\label{eq:subOpt}
\mathcal{E}(\eta) = \langle \mathbf{c}, \x^{\eta} \rangle - \min_{\x \in \mathcal{P}} \langle \mathbf{c}, \x \rangle
\end{equation}
measuring how suboptimal the solution $\x^\eta$ of \eqref{QLP} is when used as a feasible point in \eqref{LP}. It follows from the optimality of $\x^{\eta}$ for~\eqref{QLP} that $\eta\mapsto \mathcal{E}(\eta)$ is nonincreasing. (Figure~\ref{fig:projection} illustrates that it need not be strictly decreasing even on $[0,\eta^{*}]$). The optimality of $\x^{\eta}$ also implies that  $\mathcal{E}(\eta)\leq  \eta^{-1}(\|\x^*\|^2-\|\x^\eta\|^2)$; in fact, an analogous result holds for any regularization. The following improvement is particular to the quadratic penalty and will be important for our main result. 

\begin{Lemma}\label{le:hyperbolicDecay} Let $\x^\eta$ be the unique minimizer of \eqref{QLP} and let $\x^*$ be any minimizer of \eqref{LP}.  
Then 
$$
 \mathcal{E}(\eta)\leq  \frac{ \|\x^*\|^2-\|\x^\eta\|^2- \|\x^*-  \x^\eta \|^2  }{\eta}    \quad \text{for all } \eta>0. 
$$   
\end{Lemma}

\begin{Remark}\label{rk:hyperbolicDecay}
 The bound in Lemma~\ref{le:hyperbolicDecay} cannot be improved in general. Indeed, consider the example $\cP=[0,1]$ and $\c=-1$. Then $\x^{*}=1$ and $\x^{\eta}=\eta/2$ for $\eta\in(0,2]$, whereas $\x^{\eta}=\x^{*}$ for $\eta\geq2$. It is straightforward to check that the inequality in Lemma~\ref{le:hyperbolicDecay} is an equality for all $\eta>0$.
  
\end{Remark}

The next lemma details the piecewise linear nature of the curve $\eta\mapsto\x^{\eta}$. This result is known (even for some more general norms, see \cite{FinzelLi.00} and the references therein), and so is the stationary convergence \cite[Theorem~2.1]{Mangasarian.84}. For completeness, we detail a short proof in Section~\ref{se:proofs}.

\begin{Lemma}\label{le:pieceWiseAffine}
  Let $\x^\eta$ be the unique minimizer of \eqref{QLP}. The curve $\eta\mapsto\x^{\eta}$ is piecewise linear and converges stationarily to $\x^*=\arg\min_{\x\in  \cM } \|\x\|^2$ as $\eta\to\infty$. That is, there exist $n\in \N$ and  $$0=\eta_{0}<\eta_{1}<  \dots < \eta_{n}=: \eta^*$$ such that $[\eta_{i}, \eta_{i+1}]\ni \eta\mapsto \x^\eta$ is affine for every $i\in \{0, \dots, n-1\}$, and moreover, 
        $$\x^{\eta}=\x^* \quad \mbox{for all }\eta\geq \eta^{*}.$$ 
Correspondingly, the suboptimality $\mathcal{E}(\eta)=  \langle \c, \x^{\eta}-\x^{*} \rangle $ is also piecewise linear and converges stationarily to zero.
\end{Lemma}

We can now state our main result for regime of small regularization: the threshold $\eta^{*}$ beyond which $\x^{\eta}=\x^*$ and a bound for the slope of the suboptimality $\mathcal{E}(\eta)$ of~\eqref{eq:subOpt} 
before the threshold. See Figures~\ref{fig:Motivatingexample} and~\ref{fig:projection} for illustrations. We recall that $\cM$ denotes the set of minimizers of \eqref{LP} and $\ep (\mathcal{P})$ denotes the extreme points of $\mathcal{P}$. 

\begin{Theorem}\label{th:main}
Let $\x^\eta$ be the unique minimizer of \eqref{QLP} and let $\x^*$ be the minimizer of \eqref{LP} with minimal norm, $\x^*=\arg\min_{\x\in  \cM } \|\x\|^2$. Let $0=\eta_{0}<\eta_{1}<  \dots < \eta_{n}=\eta^*$ be the breakpoints of the curve $\eta\mapsto\x^{\eta}$ as in Lemma~\ref{le:pieceWiseAffine}; in particular, $\eta^*$ is the threshold such that 
$\x^{\eta}=\x^*$ for all $\eta\geq \eta^{*}$.
  
\begin{enumerate}[label=(\alph*)]
        \item\label{CriticalPoint} The threshold $\eta^{*}$ is given by 
        \begin{align}\label{eq:critEta}
         \eta^*=   2\, \max_{ \x\in\ep (\mathcal{P})\setminus \cM } \frac{\left\langle  \x^*,  \x^*-\x \right\rangle}{\left\langle {\c},  \x-\x^* \right\rangle}.
        \end{align}
        The right-hand side attains its maximum on the set $\cM (\mathcal{P}, \c^*)$ of minimizers for the linear program~\eqref{LP} with the auxiliary cost  $ \c^*:=  \frac{\eta^* \c}{2} +\x^*$. Moreover, we have $\x^{\eta}\in  \cM (\mathcal{P}, \c^*)$ for all  $\eta\in [\eta_{n-1},\eta^*]$, so that $\eta^*=   2 \frac{\left\langle  \x^*,  \x^*-\x^{\eta} \right\rangle}{\left\langle {\c},  \x^{\eta}-\x^* \right\rangle}$ for all  $\eta\in [\eta_{n-1},\eta^*]$.
               
        \item\label{SlopeBoundInThm} The slope $\frac{\mathcal{E}(\eta)}{(\eta^*-\eta) }$ of the last segment of the curve $\eta\mapsto\mathcal{E}(\eta)$ satisfies the bound 
        \begin{align}\label{eq:SlopeBoundInThm}
        \frac{\mathcal{E}(\eta)}{(\eta^*-\eta) }
        \leq  \frac12 \left\langle \c, \frac{\x^*  -\x^{\eta_{n-1}}}{\|\x^*- \x^{\eta_{n-1}}\|}\right\rangle ^{2} 
        \leq \frac{\|\c\|^2}{2}, \qquad \eta\in  [\eta_{n-1}, \eta^*).
        \end{align}
    \end{enumerate}
\end{Theorem}

It is worth noting that the first bound in~\eqref{eq:SlopeBoundInThm} is in terms of the \emph{angle} between~$\c$ and $\x^* - \x^{\eta_{n-1}}$. 
The formula~\eqref{eq:critEta} for $\eta^{*}$ is somewhat implicit in that it refers to $\x^{*}$. The following corollary states a bound for $\eta^{*}$ using similar quantities as \cite{Weed.18} uses for entropic regularization. In particular, 
we define the \emph{suboptimality gap} of $\mathcal{P}$  as 
$$\Delta := \min_{\x\in\ep(\mathcal{P}) \setminus \cM }\langle \c, \x  \rangle- \min_{\x\in \mathcal{P}}\langle \c, \x  \rangle = \min_{\x\in\ep(\mathcal{P}) \setminus \cM }\langle \c, \x - \x^{*}  \rangle ;$$
it measures the cost difference between the suboptimal and the optimal vertices of~$\cP$.

\begin{Corollary}\label{co:suboptgap}
Let $B=\sup_{\x\in\mathcal{P}} \|\x\|$ and $D=\sup_{\x,\x'\in\mathcal{P}} \|\x-\x'\|$ be the bound and diameter of $\mathcal{P}$, respectively.
Then 
$$
  \eta^{*} \leq \frac{2BD}{\Delta}.
$$
\end{Corollary} 

For integer programs, where $\c$ and the vertices of $\cP$ have integer coordinates, it is clear that $\Delta\geq1$. In general, the explicit computation of $\Delta$ is not obvious.  In Section~\ref{se:OT} below we shall find it more useful to directly use~\eqref{eq:critEta}.

We conclude this section with a quantitative result for  the regime $\eta\to 0$ of large regularization. After rescaling with $\eta$, the quadratically regularized linear program \eqref{QLP} formally tends to the quadratic program 
\begin{align}{\tag{QP}}\label{QP}
    \text{minimize}~~  {\|\x\|^2} \quad
    \text{subject to}~~ \x\in \mathcal{P}.
\end{align}
The unique solution $\x^0$ of~\eqref{QP} is simply the projection of the origin onto~$\cP$. It is known in several contexts that $\x^{\eta}\to\x^0$ as $\eta\to 0$ (e.g., \cite[Properties~2,7]{DesseinPapadakisRouas.18}). The following result quantifies this convergence by establishing that $\| \x^{\eta}-\x^0\|$ tends to zero at a linear rate.

\begin{Proposition}\label{pr:largeRegularization}
  Let $\x^\eta$ and $\x^0$ be the minimizers of \eqref{QLP} and \eqref{QP}, respectively. Then
  \begin{equation*}%
    \| \x^{\eta}-\x^0\|\leq \frac{1}{2} \|\c\| \eta   \quad \text{for all } \eta>0.
    \end{equation*}
\end{Proposition}

\begin{Remark}\label{rk:largeRegularizationSharp}
  The bound in Proposition~\ref{pr:largeRegularization} is sharp in the example $\cP=[0,1]$ and $\c=-1$. %
\end{Remark} 

\begin{Remark}\label{rk:largeRegularizationGeneral}
  Proposition~\ref{pr:largeRegularization} and its proof apply to an arbitrary closed, bounded convex set~$\cP$ in a Hilbert space, not necessarily a polytope. In particular, the bounds also hold for continuous optimal transport problems.
\end{Remark}

\section{Application to Optimal Transport}\label{se:OT}

Recall from the Introduction the optimal transport problem with cost function $c(\cdot,\cdot)$ between probability measures $\mu$ and $\nu$,
\begin{equation}\tag{OT}\label{OTCONt2}
\inf_{\gamma\in \Gamma(\mu, \nu)} \int c(x,y) d\gamma(x,y), 
 \end{equation}
where $\Gamma(\mu, \nu)$ denotes the set of couplings of $(\mu, \nu)$, and its quadratically regularized version
\begin{equation}
    {\tag{QOT}}\label{QOT2}
    \inf_{\gamma\in \Gamma(\mu, \nu)} \int c(x,y)d\gamma(x,y)+\frac{1}{\eta}  \left\| \frac{d\gamma}{  d(\mu \otimes \nu) }\right\|_{L^2( \mu \otimes \nu)}^2.
\end{equation} 
Throughout this section, we consider given points $\X_i,\Y_{i}$, $1\leq i\leq N$ (in $\R^{D}$, say) with their associated empirical measures and cost matrix
$$
\mu=\frac{1}{N}\sum_{i=1}^N \delta_{\X_i}, \qquad \nu=\frac{1}{N}\sum_{i=1}^N \delta_{\Y_i},\qquad C_{ij}:=c(\X_{i},\X_{j}).
$$
Any coupling $\gamma$ gives rise to a matrix $\gamma_{ij}=\gamma(\X_{i},\Y_{j})$ through its probability mass function. Those matrices form the set 
$$
  \Gamma_{N}=\{\gamma\in\R^{N\times N}:\, \gamma\, {\bf 1}  = N^{-1}{\bf 1},  \; \gamma^\top\, {\bf 1}=N^{-1} {\bf 1}, \; \gamma_{i,j}\geq 0\}.
$$
It is more standard to work instead with the \emph{Birkhoff polytope} of doubly stochastic matrices,
$$
  \Pi_{N}=\{\pi\in\R^{N\times N}:\, \pi\, {\bf 1}  =  {\bf 1},  \; \pi^\top\, {\bf 1} = {\bf 1}, \; \pi_{i,j}\geq 0\},
$$
that is obtained through the bijection $\pi_{ij}=N\gamma_{ij}$. 
By Birkhoff's theorem (e.g., \cite{Brualdi.06}), the extreme points $\ep(\Pi_{N})$ are precisely the permutation matrices; i.e., matrices with binary entries whose rows and columns sum to one. 
Let $\langle A, B \rangle  :={\rm Trace}(A^\top B)=\sum_{i=1}^N \sum_{j=1}^N A_{i,j}B_{i,j} $ be the Frobenius inner product on $\R^{N\times N}$ and $\| \cdot \| $ the associated norm. Then \eqref{QOT2} becomes a particular case of \eqref{QLP}, namely
\begin{align}\label{DQOT}
     \min_{\gamma\in\Gamma_{N}}  \langle C, \gamma \rangle +\frac{N^2}{\eta}\|\gamma\|^2  
     \qquad\text{or equivalently}\qquad
    \min_{\pi\in\Pi_{N}} \frac{1}{N}\langle C, \pi \rangle +\frac{1}{\eta}\|\pi\|^2,
\end{align}
where the factor $N^{2}$ is due to $\mu \otimes \nu$ being the uniform measure on $N^{2}$ points. 
To have the same form as in~\eqref{QLP} and Section~\ref{se:main}, we write~\eqref{DQOT} as
\begin{align}
    \label{BirkoffOT}
    \min_{\pi\in\Pi_{N}} \langle {\bf c}, \pi \rangle +\frac{1}{\eta}\|\pi\|^2  \qquad \text{where }\c_{ij}:=C_{ij}/N.
\end{align}
We can now apply the general results of Theorem~\ref{th:main} to~\eqref{BirkoffOT} and infer the following for the regularized optimal transport problem~\eqref{QOT2}; a detailed proof can be found in  Section~\ref{se:proofs}.

\begin{Proposition}\label{pr:OT}
\begin{enumerate}[label=(\alph*)]
    \item The optimal coupling $\gamma^{\eta}$ of \eqref{QOT2} is optimal for \eqref{OTCONt2} if and only if    \begin{equation}
        \label{etaEstrellaOT}
        \eta \geq \eta^*:=2\, N \cdot \max_{ \pi\in \ep (\Pi_{N})\setminus \cM} \frac{\left\langle  \pi^*,  \pi^*-\pi \right\rangle }{\left\langle C,  \pi-\pi^* \right\rangle },
    \end{equation}
    in which case $\gamma^{\eta}$ is the minimum-norm solution $\gamma^{*}$ of \eqref{OTCONt2}.  
    
    \item We have the following bound for the slope of the suboptimality,
    \begin{multline}
         \label{eq:SlopeOT}
        \limsup_{\eta\to \eta^*} \frac{\int c(x, y) d\gamma^{\eta}(x, y)- \int c(x, y) d\gamma^{*}(x, y)}{\eta^*-\eta}\\
        \leq \frac{1}{2}\left( \int c(x, y)^2 d(\mu\otimes \nu)(x, y)-\left(\int c(x, y) d(\mu\otimes \nu)(x, y)\right)^2\right).
    \end{multline}
\end{enumerate}    
\end{Proposition}

The following example shows that Proposition~\ref{pr:OT} is sharp.
 
\begin{Example}\label{ex:ExampleOTEasy} Let  $c(\X_i, \Y_j)=-\delta_{ij}$, so that $\pi^{*}=\id$ is the  identity matrix and $C=-\id$. Note also that $\pi^0$ has entries $\pi_{i,j}^0=1/N$. 
It follows from~\eqref{etaEstrellaOT} that $\eta^*=2 N$, and the right-hand side of~\eqref{eq:SlopeOT} evaluates to $\frac{N-1}{2N^{2}}$. We show below that $[0,\eta^{*}]\ni\eta\mapsto \x^{\eta}$ is affine, or more explicitly, that $\pi^\eta= \frac{2N-\eta}{2N}\pi^0 +\frac{\eta}{2N} \pi^*=:\Tilde{\pi}^\eta.$ 
As a consequence, we have for every $ \eta\in [0, \eta^*) $ that
\begin{align*}
    \frac{\int c(x, y) d\gamma^{\eta}(x, y)- \int c(x, y) d\gamma^{*}(x, y)}{\eta_*-\eta}& = \frac{\langle C,  \pi^\eta-\pi^*\rangle }{N(\eta_*-\eta)} 
    =-\frac{ (2N-\eta)+ (\eta-2N)N }{2N^2(\eta_*-\eta)}\\
    &=- \frac{ (\eta^*-\eta)+ (\eta-\eta^*)N }{2N^2(\eta_*-\eta)}
    =\frac{ N-1 }{2N^2},
\end{align*}
matching the right-hand side of~\eqref{eq:SlopeOT}.

It remains to show that $\pi^\eta=\Tilde{\pi}^\eta$. Using $\c= \id/N$, the definition of $\Tilde{\pi}^\eta$ and $\pi^{*}=\id$, we see that 
$ \frac{\eta \c}{2} + \Tilde{\pi}^\eta = \frac{2N-\eta}{2N}\pi^0.$ The form of $\pi^0$ also implies that $\left\langle \pi^{0},  \pi'-\pi\right\rangle =0 $ for any $\pi,\pi'\in \Pi_N$. Together, it follows that 
$
  \left\langle -\frac{\eta \c}{2} - \Tilde{\pi}^\eta,  \Tilde{\pi}^\eta -\pi\right\rangle =0 $
for all $\pi\in \Pi_N$. By Lemma~\ref{Lemma:projection}, this implies $\Tilde{\pi}^\eta=\pi^\eta$.
\end{Example}

Next, we focus on a more representative class of transport problems. Our main interest is to see how our key quantities scale with $N$, the number of data points.

\begin{Corollary}
   \label{Coro:OTseparated}
   Assume that there exist $[\epsilon_m, \epsilon_M] \subset [0,\infty) $ and a permutation $\sigma^*: \{1, \dots, N\}\to \{1, \dots, N\}$  such that 
    $$ \kappa:=\min_{i\in \{1, \dots, N\}, j\neq \sigma^*(i)}c(\X_i, \Y_j)>\epsilon_M\quad \text{and}\quad c(\X_i, \Y_{\sigma^*(i)})\in [\epsilon_m, \epsilon_M]  \quad \text{for all $i\in \{1, \dots, n\}$}.$$
      Then 
      \begin{equation}\label{eq:OTseparatedBounds}
      \frac{4\,N}{\kappa'-2\epsilon_m} \leq \eta^*\leq \frac{2\,N}{\kappa-\epsilon_M},
      \end{equation}
     where $\kappa':= \min_{i\neq j} c(\X_i, \Y_{\sigma^*(j)})+ c(\X_{j},\Y_{\sigma^*(i)}).$
    If the cost is symmetric around $\sigma^*$ in the sense that $c(\X_i, \Y_{\sigma^*(j)})= c(\X_{j},\Y_{\sigma^*(i)})$ for all $i,j\in \{1, \dots, N\} $, then 
    \begin{equation}\label{eq:OTseparatedBoundsSym}
  \frac{2\,N}{\kappa-\epsilon_m} \leq   \eta^*\leq  \frac{2\,N}{\kappa-\epsilon_M}, \qquad 
    \text{and in particular}\quad\eta^*=\frac{2\,N}{\kappa}\quad
    \text{if}\quad\epsilon_m=\epsilon_M=0.
    \end{equation}
\end{Corollary}

The proof is detailed in Section~\ref{se:proofs}.
We illustrate Proposition~\ref{pr:OT} and Corollary~\ref{Coro:OTseparated} with a representative example for scalar data.

\begin{Example}\label{ex:oneoverNOT}
Consider the quadratic cost $c(x, y)=\|x-y\|^2$ and $\X_{i}=\Y_{i}=\frac{i}{N}$, $1\leq i\leq N$ with $N\geq2$, leading to the cost matrix
$$
  C_{ij}=\frac{|i-j|^{2}}{N^{2}}.
$$
Then
$$
  \eta^{*} = 2 N^3
$$
and we have the following bound for the slope of the suboptimality,
\begin{equation}
\label{upperBoundforExample}
    \limsup_{\eta\to \eta^*} \frac{\int c(x, y) d\gamma^{\eta}(x, y)- \int c(x, y) d\gamma^{*}(x, y)}{\eta^*-\eta} \leq  \frac{N-1}{ N^6}.
\end{equation}
\end{Example} 

Indeed, the value of $\eta^{*}$ follows directly from the last part of~\eqref{eq:OTseparatedBoundsSym} with $\kappa=1/N^2$ and $\sigma^*$ being the identity. The proof of~\eqref{upperBoundforExample} is longer and relegated to Section~\ref{se:proofs}.

To study the accuracy of the bound \eqref{upperBoundforExample}, we compute numerically the limit 
$$ L_N= \lim_{\eta\to \eta^*} \frac{\int c(x, y) d\gamma^{\eta}(x, y)- \int c(x, y) d\gamma^{*}(x, y)}{\eta^*-\eta} $$
for \(N = j*30 \) with $j=2, \dots, 16$. Figure~\ref{fig:BoundSlope} shows $N\mapsto L_N$ in blue and the upper bound $N\mapsto \frac{N-1}{ N^6}$ in red (in double logarithmic scale). We observe that both have the same order as a function of \(N\).

\begin{figure}[h!]
    \centering
		\resizebox{.7\linewidth}{!}{%

\tikzset{every picture/.style={line width=0.75pt}} %

\begin{tikzpicture}[x=0.75pt,y=0.75pt,yscale=-1,xscale=1, scale=1]
\draw (330.5,114.9) node  {\includegraphics[width=362.25pt,height=217.35pt]{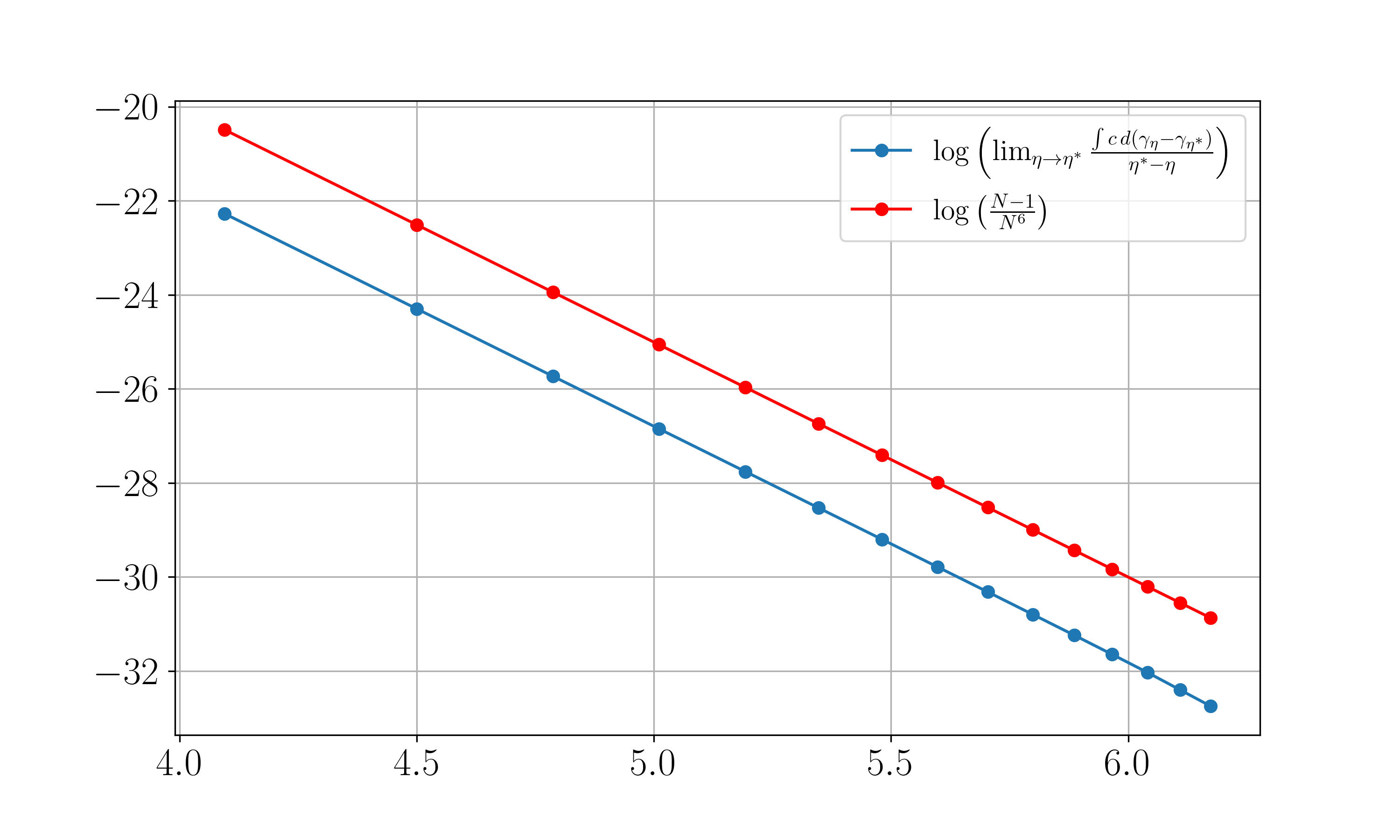}};

\draw (322,242.4) node [anchor=north west][inner sep=0.75pt]    {$\log( N)$};

\end{tikzpicture}
}%
 \caption{Accuracy of the bound~\eqref{upperBoundforExample}. Plot of \(N \mapsto \lim_{\eta\to \eta^*} \frac{\int c(x, y) d\gamma^{\eta}(x, y)- \int c(x, y) d\gamma^{*}(x, y)}{\eta^*-\eta}\) (blue)  and the upper bound \(N \mapsto \frac{N-1}{N^6}\) (red) in double logarithmic scale.
}
    \label{fig:BoundSlope}
\end{figure}

\section{Proofs}\label{se:proofs}

\begin{proof}[Proof of Lemma~\ref{le:hyperbolicDecay}] 
Let $\x\in\cP$. Inserting the definition of $\Phi_\eta$, expanding $\|\x-  \x^\eta \|^2$, and applying Lemma~\ref{Lemma:projection} yield
\begin{align*}
    \Phi_\eta (\x) & = \Phi_\eta (\x^\eta)+ \left\langle \c+ \frac{ 2\x^\eta}{\eta} , \x-  \x^\eta \right\rangle +\frac{ \|\x-  \x^\eta \|^2 }{\eta}\geq \Phi_\eta (\x^\eta)+ \frac{ \|\x-  \x^\eta \|^2 }{\eta}.
\end{align*}
Therefore, 
\begin{align*}
    0 &\geq  \Phi_\eta (\x^\eta)-  \Phi_\eta (\x) + \frac{ \|\x-  \x^\eta \|^2 }{\eta}= \langle \c,  \x^\eta-\x\rangle + \frac{ \|\x^\eta\|^2- \|\x\|^2+  \|\x-  \x^\eta \|^2 }{\eta}%
\end{align*} 
and in particular choosing $\x=\x^{*}$ gives
$$  \mathcal{E}(\eta)= \langle \c,  \x^\eta-\x^*\rangle\leq \frac{ \|\x^*\|^2-\|\x^\eta\|^2- \|\x^*-  \x^\eta \|^2  }{\eta}$$
as claimed.
\end{proof}

\begin{proof}[Proof of Lemma~\ref{le:pieceWiseAffine} and Theorem~\ref{th:main}]   \emph{Step 1.} Let $\eta_{(1)} < \eta_{(2)}$. We claim that if $\x^{\eta_{(1)}}, \x^{\eta_{(2)}} \in \text{{\rm ri}}(\mathcal{F})$ for some face\footnote{A nonempty face \(\mathcal{F}\) of the polytope \(\mathcal{P}\) can be defined as a subset  \(\mathcal{F}\subset\mathcal{P}\) such that there exists  an affine hyperplane $H=\{ \x\in \R^d: \langle \x, {\bf a} \rangle=m \} $ with $H\cap \mathcal{P}=\mathcal{F} $ and $ \mathcal{P}\subset  \{ \x\in \R^d: \langle \x, {\bf a} \rangle\leq m \}$. See \cite{Brondsted.83}.
} $\mathcal{F}$ of $\mathcal{P}$, then $[\eta_{(1)}, \eta_{(2)}] \ni \eta \mapsto \x^{\eta}$ is affine. Indeed, $\x^{\eta_{(i)}}=\proj_{\cP}(-\eta_{(i)} \c/2)$ is the projection of $-\eta_{(i)} \c/2$ onto $\cP$. As $\x^{\eta_{(i)}}\in \text{{\rm ri}}(\mathcal{F})$, it follows that $\x^{\eta_{(i)}}=\proj_{A}(-\eta_{(i)} \c/2)$ is also the projection onto the affine hull~$A$ of $\cF$. Since $A$ is an affine space, the map $\eta\mapsto\proj_{A}(-\eta\c/2)$ is affine. For $\eta_{(1)}\leq \eta\leq  \eta_{(2)}$, convexity of $\text{{\rm ri}}(\mathcal{F})$ then implies $\proj_{A}(-\eta\c/2)\in \text{{\rm ri}}(\mathcal{F})$ , which in turn implies $\proj_{A}(-\eta\c/2)=\proj_{\cF}(-\eta\c/2)=\proj_{\cP}(-\eta\c/2)=\x^{\eta}$.

\emph{Step 2.} We can now define $\eta_1, \dots, \eta_{n}$ recursively as follows. Recall first that each $\x\in\cP$ is in the relative interior of exactly one face of~$\cP$ (possibly~$\cP$ itself), namely the smallest face containing~$\x$ \cite[Theorem~5.6]{Brondsted.83}. Let $\mathcal{F}_0$ be the unique face such that $ \x^0:=\argmin_{\x\in\cP}\|\x\|\in {\rm ri}(\mathcal{F}_0)$ and define 
$$
\eta_1 := 
    \inf\{ \eta > 0: \ \x^\eta \notin  \text{{\rm ri}}(\mathcal{F}_0) \},
$$
where we use the convention that $\inf \emptyset  =+\infty$. Then  $(0, \eta_{1}) \ni \eta \mapsto \x^\eta$ is affine by Step~1. For $i>1$, if $\eta_{i-1} < \infty$, let $\mathcal{F}_{i-1}$ be the face such that $\x^{\eta_{i-1}}\in \text{{\rm ri}}(\mathcal{F}_{i-1})$ and define 
$$
\eta_{i} := 
    \inf\{ \eta > \eta_{i-1}: \ \x^\eta \notin \text{{\rm ri}}(\mathcal{F}_{i-1}) \}.
$$
Again, $(\eta_{i-1}, \eta_{i}) \ni \eta \mapsto \x^\eta$ is affine by Step~1. Moreover, by continuity, $[\eta_{i-1}, \eta_{i}] \ni \eta \mapsto \x^\eta$ is also affine. 

\emph{Step 3.} Next, we establish the value~\eqref{eq:critEta} of $\eta^{*}$. Let us first observe that~\eqref{eq:critEta} is strictly positive. Indeed, the denominator is clearly positive. Suppose that the numerator $\left\langle \x^*,  \x-\x^* \right\rangle \leq 0$ for all $\x\in  \ep (\mathcal{P})\setminus \cM$. Note that by the definition of $\x^*$, we also have $\left\langle \x^*,  \x-\x^* \right\rangle \leq 0$ for all $\x\in  \cM$. Thus $\left\langle \x^*,  \x-\x^* \right\rangle \leq 0$ for all $\x\in  \ep (\mathcal{P})$, meaning that $\x^*$ is the projection of the origin onto $\mathcal{P}$, the degenerate situation we had excluded in our setup.

Let $\eta >0$ and suppose that 
$\x^*=\x^{\eta}$. Then by Lemma~\ref{Lemma:projection},
    \begin{equation*}
    \label{caracterizedXstar}
    -\left\langle \x^*,  \x-\x^* \right\rangle \leq \left\langle \frac{\eta \c}{2},  \x-\x^* \right\rangle   \quad \text{for all } \x\in \mathcal{P}  .
    \end{equation*}
Using also that $\left\langle {\c},  \x-\x^* \right\rangle> 0$ for $\x\in \mathcal{P} \setminus \cM $, we deduce
 \begin{equation}
     \label{toreverse}
     \eta \geq 2\frac{\left\langle \x^*,  \x^*-\x \right\rangle}{\left\langle {\c},  \x-\x^* \right\rangle}  \quad \text{for all } \x\in  \ep (\mathcal{P})\setminus \cM .   
 \end{equation}
Conversely, assume that \eqref{toreverse} holds; we show that $\x^*=\x^{\eta}$. Recall that $\ep (\mathcal{P})= \{\v_1, \dots, \v_{K}\} $ denotes the set of extreme points of $\mathcal{P}$. Let  $\x\in \mathcal{P}$, then there exist $\{\lambda_i\}_{i=1}^K \subset [0,1]$ with $1=\sum_{i=1}^K \lambda_i$ such that $\x=\sum_{i=1}^K \lambda_i\v_i $. 
We note that \eqref{toreverse} yields
$$  \left\langle \frac{\eta \c}{2},  \x-\x^* \right\rangle = \sum_{i:\, \v_i\in \ep(\cP)\setminus \ep(\cM)}\lambda_i \left\langle \frac{\eta \c}{2},  \v_i-\x^* \right\rangle \geq -\sum_{i:\, \v_i\in \ep(\cP)\setminus \ep(\cM)}\lambda_i \left\langle \x^*,  \v_i-\x^* \right\rangle .  $$ 
On the other hand, the fact that $\x^*$ is the projection of the origin onto $\cM$ yields
$$ \sum_{i:\, \v_i\in \ep(\cM)} \lambda_i \left\langle \x^*,  \v_i-\x^* \right\rangle \geq 0.$$
Together,
\begin{align*}
    \left\langle \frac{\eta \c}{2},  \x-\x^* \right\rangle &\geq -\sum_{i:\, \v_i\in \ep(\cP)\setminus \ep(\cM)}\lambda_i \left\langle \x^*,  \v_i-\x^* \right\rangle \geq -\sum_{i:\, \v_i\in \ep(\cP)}\lambda_i \left\langle \x^*,  \v_i-\x^* \right\rangle 
    =  -\left\langle \x^*,  \x-\x^* \right\rangle .
\end{align*}
As $\x\in \mathcal{P}$ was arbitrary, Lemma~\ref{Lemma:projection} now shows that $\x^*=\x^{\eta}$. This completes the proof of Lemma~\ref{le:pieceWiseAffine} and~\eqref{eq:critEta}.

Finally, note that $\x$ attains the maximum in~\eqref{eq:critEta} if and only if $\langle \c^{*},  \x-\x^* \rangle=0$. Moreover, $\langle \c^{*},  \x-\x^* \rangle\geq0$ for all $ \x\in\cP$ by Lemma~\ref{Lemma:projection}. Hence the set of maximizers of~\eqref{eq:critEta} over $\ep (\mathcal{P})\setminus \cM$ equals the set of minimizers of $\langle \c^{*},  \cdot \rangle$ over $\ep(\mathcal{P})$.
        
\emph{Step 4.} We prove the remaining claim in~\ref{CriticalPoint}, namely that  $\x^{\eta}\in  \cM (\mathcal{P}, \c^*)$ for all  $\eta\in [\eta_{n-1},\eta^*]$. By Lemma~\ref{Lemma:projection}, 
$$ \left\langle -\frac{\eta \c}{2}-\x^{\eta},  \x-\x^{\eta} \right\rangle \leq 0   \quad \text{for all } \x\in \mathcal{P}, \ \eta\in [\eta_{n-1}, \eta^*].$$
As
$\x^{\eta}\in {\rm ri}([\x^{\eta_{n-1}},\x^*] )$ for $ \eta\in (\eta_{n-1}, \eta^*)$, Lemma~\ref{Lemma:projection} moreover yields
$$ \left\langle -\frac{\eta \c}{2}-\x^{\eta},  \x^{\eta'}-\x^{\eta} \right\rangle =0   \quad \text{for all } \eta'\in  [\eta_{n-1}, \eta^*], \  \eta\in (\eta_{n-1}, \eta^*), $$
and by continuity, the previous display also holds for $\eta\in [\eta_{n-1}, \eta^*]$. In summary, we have  
\begin{equation}\label{eq:step4}
\left\langle -\frac{\eta^* \c}{2}-\x^*,  \x-\x^* \right\rangle \leq 0 \quad\text{for all } \x\in \mathcal{P}
\end{equation}
and 
$$ \left\langle -\frac{\eta^* \c}{2}-\x^*,  \x^{\eta_{n-1}}-\x^* \right\rangle =0    . $$
Therefore, $\x^{\eta_{n-1}}\in \cM (\mathcal{P}, \c^*)$. On the other hand, \eqref{eq:step4} also states that $\x^{\eta^{*}}=\x^{*}\in \cM (\mathcal{P}, \c^*)$, and then convexity implies the claim. %

\emph{Step 5.} It remains to prove \ref{SlopeBoundInThm}. Let $\eta\in (\eta_{n-1}, \eta^*)$. Then Lemma~\ref{le:pieceWiseAffine} implies that $\x^\eta= \lambda \x^{\eta_{n-1}} + (1-\lambda) \x^* $ for some $\lambda\in (0, 1)$ and thus 
$$ \langle \c , \x^\eta \rangle = \langle \c , \x^* \rangle +\lambda \langle \c, \x^{\eta_{n-1}}-\x^* \rangle.$$
Lemma~\ref{le:hyperbolicDecay} then yields 
$$ \lambda =  \frac{ \langle \c , \x^\eta - \x^* \rangle  }{\langle \c, \x^{\eta_{n-1}}-\x^* \rangle} \leq \frac{ \|\x^*\|^2-\|\x^\eta\|^2- \|\x^*-  \x^\eta \|^2    }{\eta \langle \c, \x^{\eta_{n-1}}-\x^* \rangle}.   $$
Using 
$$
    \|\x^\eta\|^2 = \|\x^*\|^2+\lambda^2 \|\x^*- \x^{\eta_{n-1}}\|^2+ 2\lambda \langle \x^*, \x^{\eta_{n-1}}- \x^*\rangle   
$$
and 
$
    \|\x^\eta-\x^*\|^2 = \lambda^2 \|   \x^*-  \x^{\eta_{n-1}} \|^2 ,
$ 
it follows that
$$ \lambda  \leq \frac{ 2\lambda \langle \x^*,  \x^*-\x^{\eta_{n-1}} \rangle-2\lambda^2 \|\x^*- \x^{\eta_{n-1}}\|^2   }{\eta \langle \c, \x^{\eta_{n-1}}-\x^* \rangle}.   $$ 
and hence
\begin{align}\label{eq:stepToSlope}
    \lambda &\leq  \frac{2\langle \x^*,  \x^*-\x^{\eta_{n-1}} \rangle-\eta \langle \c, \x^{\eta_{n-1}}-\x^* \rangle}{2 \|\x^*- \x^{\eta_{n-1}}\|^2}.
\end{align}
By the last part of \ref{CriticalPoint} we have
\begin{align*}
\eta^* =  \frac{2\left\langle  \x^*,  \x^*-\x^{\eta_{n-1}} \right\rangle}{\left\langle {\c},  \x^{\eta_{n-1}}-\x^* \right\rangle}.
\end{align*} 
Inserting this in~\eqref{eq:stepToSlope} yields
\begin{align*}
    \lambda  \leq   \frac{(\eta^*-\eta) \langle \c, \x^{\eta_{n-1}}-\x^* \rangle }{2 \|\x^*- \x^{\eta_{n-1}}\|^2} 
\end{align*}
and now it follows that
$$ \mathcal{E}(\eta)= \lambda \langle  \c, \x^{\eta_{n-1}}-\x^* \rangle \leq \frac{(\eta^*-\eta) \langle \c, \x^{\eta_{n-1}}-\x^* \rangle ^2 }{2 \|\x^*- \x^{\eta_{n-1}}\|^2} $$
as claimed.
\end{proof}

\begin{proof}[Proof of Proposition~\ref{pr:largeRegularization}]
    Recall that $\x^\eta$ is the projection of $-\eta \c/2$ onto~$\cP$ whereas $\x^0$ is the projection of the origin onto~$\cP$. As the projection operator onto a convex set is non-expanding (i.e., Lipschitz continuous with constant one), this implies $\| \x^{\eta}-\x^0\|\leq \|-\eta \c/2\|=\frac{1}{2} \|\c\| \eta$.
\end{proof}

\begin{proof}[Proof of Proposition~\ref{pr:OT}.]
  Theorem~\ref{th:main}\ref{CriticalPoint} directly yields \eqref{etaEstrellaOT}. Whereas for~\eqref{eq:SlopeOT}, a direct application of Theorem~\ref{th:main}\ref{SlopeBoundInThm} only yields
  $$ \limsup_{\eta\to \eta^*} \frac{\int c(x, y) d\gamma^{\eta}(x, y)- \int c(x, y) d\gamma^{*}(x, y)}{\eta^*-\eta} \leq \frac{1}{2}\int c(x, y)^2 d(\mu\otimes \nu)(x, y). $$
  To improve this bound, note that the optimizer of \eqref{QOT2} does not change if the cost $c(x, y)$ is changed by an additive constant. Moreover, for any $m\in \R $,
    \begin{multline*}
      {\int c(x, y) d\gamma^{\eta}(x, y)- \int c(x, y) d\gamma^{*}(x, y)} 
       ={\int (c(x, y)-m) d\gamma^{\eta}(x, y)- \int (c(x, y)-m) d\gamma^{*}(x, y)}. 
  \end{multline*}
Applying Theorem~\ref{th:main} with the modified cost $c(x, y)-m$ for the choice $m:=\int c(x, y) d(\mu\otimes \nu)(x, y)$ yields \eqref{eq:SlopeOT}. 
\end{proof}

\begin{proof}[Proof of Corollary~\ref{Coro:OTseparated}]
Assume without loss of generality that $ \sigma^*$ is the identity, so that  $\pi^*={\rm Id}$ is the identity matrix.  
    Let $P_\sigma$ be the permutation matrix associated with a permutation $\sigma: \{1, \dots, N\}\to \{1, \dots, N\}$. We define 
    $\mathcal{N}(\sigma)=\{ i\in \{1, \dots, N\}:\ \sigma(i)=i   \}.$ Then 
    \begin{equation}\label{eq:OTseparatedProof}
    \frac{\left\langle  \pi^*,  \pi^*-P_\sigma \right\rangle }{\left\langle C,  P_\sigma-\pi^* \right\rangle } = \frac{N-\vert \mathcal{N}(\sigma)\vert  }{{\sum_{i\notin \mathcal{N}(\sigma)} c(\X_i, \Y_{\sigma(i)})-c(\X_i, \Y_{i})} },
    \end{equation}
    where $\vert \mathcal{N}(\sigma)\vert$ denotes the cardinality of $\mathcal{N}(\sigma)$. 
    
    For the upper bound in~\eqref{eq:OTseparatedBounds}, we recall that $c(\X_i, \Y_{i})\leq \epsilon_M$ and $c(\X_i, \Y_{\sigma(i)})\geq\kappa$ for $i\notin \mathcal{N}(\sigma)$, so that~\eqref{eq:OTseparatedProof} yields
    $$ \frac{\left\langle  \pi^*,  \pi^*-P_\sigma \right\rangle }{\left\langle C,  P_\sigma-\pi^* \right\rangle } \leq  \frac{1 }{\kappa-\epsilon_M} \frac{N-\vert \mathcal{N}(\sigma)\vert  }{N-\vert \mathcal{N}(\sigma)\vert  }=\frac{1 }{\kappa-\epsilon_M}. $$
    Now Proposition~\ref{pr:OT} yields the claim.     For the lower bound in~\eqref{eq:OTseparatedBounds}, let $i^*,j^*\neq \sigma^*(i^*)$ be such that  $\kappa'=c(\X_{i^*},\Y_{j^*})+c(\X_{j^*}, \Y_{i^*})$ and let $\sigma$ be the permutation such that   $\sigma(i)=i$ for all $i\notin \{i^*, j^*\}$,  $\sigma(i^*)=j^*$ and $\sigma(j^*)=i^*$. Then 
    $$ \frac{\left\langle  \pi^*,  \pi^*-P_\sigma \right\rangle }{\left\langle C,  P_\sigma-\pi^* \right\rangle }=  \frac{2}{ c(\X_{i^*}, \Y_{j^*}) +c(\X_{j^*}, \Y_{i^*})- (c(\X_{i^*}, \Y_{i^*}) + c(\X_{j^*}, \Y_{j^*}))}\geq  \frac{2}{  \kappa'-2\epsilon_m}$$
    and Proposition~\ref{pr:OT} again yields the claim. It remains to observe that $\kappa'=2\kappa$ when the cost is symmetric. 
\end{proof}

\begin{proof}[Proof for Example~\ref{ex:oneoverNOT}]
Corollary~\ref{Coro:OTseparated} applies with \(\sigma^*\) being the identity and \(\kappa=1/N^2\). As a consequence, the  critical value \(\eta^*\) is $2 N^3.$

To prove~\eqref{upperBoundforExample}, write $\pi^{\eta_{n-1}}=  \sum_{i=1}^k \lambda_i P_{\sigma_i}$ with $\lambda_i\in (0,1]$  and $\sum_{i=1}^k \lambda_i=1$. Recall from Theorem~\ref{th:main}\ref{CriticalPoint} that $0=\left\langle \c^*, \pi^{\eta_{n-1}} -\pi^* \right\rangle $. With the optimality of $\pi^{*}=\pi^{\eta^{*}}$ for $\left\langle \c^*, \cdot\right\rangle$, this implies
$$ 0=\left\langle \c^*, P_{\sigma_i} -\pi^* \right\rangle =\left\langle \frac{\eta^*C}{2\, N}+\pi^*, P_{\sigma_i} -\pi^* \right\rangle =\left\langle {N^2C}+\pi^*, P_{\sigma_i} -\pi^* \right\rangle  \quad \text{for all } i=1, \dots, k.$$
As $\langle {N^2C}+\pi^*,\pi^*\rangle = \langle N^2C + \id ,\id\rangle = N$, it follows that  
$ \left\langle {N^2C} + \id, P_{\sigma_i} \right\rangle=N $. Using that $P_{\sigma_i}$ has~$N$ entries equal to one and that the entries of $N^2C + \id$ are strictly larger than one outside the three principal diagonals, this implies that $|\sigma_i(j)-j|\leq 1$ for all $j\in \{1, \dots, N\}$.  As a consequence,  
$\pi^{\eta_{n-1}}=  \sum_{i=1}^k \lambda_i P_{\sigma_i}$ vanishes outside the three principal diagonals; i.e., it is entry-wise smaller or equal to the tridiagonal matrix 
$$ A=\begin{pmatrix}
1 & 1 & 0 & 0 & \cdots & 0 & 0 \\
1 & 1 & 1 & 0 & \cdots & 0 & 0 \\
0 & 1 & 1 & 1 & \cdots & 0 & 0 \\
0 & 0 & 1 & 1 & \cdots & 0 & 0 \\
\vdots & \vdots & \vdots & \vdots & \ddots & \vdots & \vdots \\
0 & 0 & 0 & 0 & \cdots & 1 & 1 \\
0 & 0 & 0 & 0 & \cdots & 1 & 1 \\
\end{pmatrix}. $$
Let $\bar\c:=A\odot\c$ be the entry-wise product, meaning that entries of $\c$ outside the three principal diagonals are set to zero. As $\pi^{\eta_{n-1}}-{\rm Id}$ vanishes outside those diagonals, we have 
$$\langle \pi^{\eta_{n-1}}-{\rm Id}, \c\rangle  = \langle \pi^{\eta_{n-1}}-{\rm Id}, \bar{\c}\rangle .
$$
We can now use Theorem~\ref{th:main}\ref{SlopeBoundInThm} and the Cauchy--Schwarz inequality to find 
\begin{align*}
 \limsup_{\eta\to \eta^*} \frac{ \langle \pi^{\eta}-{\rm Id}, \c\rangle  }{(\eta^* -\eta) }
\leq  \frac{\langle \pi^{\eta_{n-1}}-{\rm Id}, \c\rangle^2 }{2 \|\pi^{\eta_{n-1}}-{\rm Id}\|^2 }
 & = \frac{\langle \pi^{\eta_{n-1}}-{\rm Id}, \bar{\c}\rangle^2 }{2 \|\pi^{\eta_{n-1}}-{\rm Id}\|^2 } 
 \leq \frac{\|\bar\c\| ^{2}}{2} = \frac{1}{2N^{2}} \frac{2(N-1)}{N^{4}} =\frac{N-1}{ N^6}
\end{align*} 
as claimed in~\eqref{upperBoundforExample}.
\end{proof} 

\paragraph{Statements and Declarations} M.\ Nutz was partially supported by NSF Grants DMS-1812661, DMS-2106056, DMS-2407074. The authors have no relevant financial or non-financial interests to disclose. All authors have contributed to all parts of the paper. %

\newcommand{\dummy}[1]{}


\begin{thebibliography}{10}

\bibitem{AguehCarlier.11}
M.~Agueh and G.~Carlier.
\newblock Barycenters in the {W}asserstein space.
\newblock {\em SIAM J. Math. Anal.}, 43(2):904--924, 2011.

\bibitem{AltschulerNilesWeedStromme.21}
J.~M. Altschuler, J.~Niles-Weed, and A.~J. Stromme.
\newblock Asymptotics for semidiscrete entropic optimal transport.
\newblock {\em SIAM J. Math. Anal.}, 54(2):1718--1741, 2022.

\bibitem{AlvarezJaakkola.18}
D.~Alvarez-Melis and T.~Jaakkola.
\newblock Gromov-{W}asserstein alignment of word embedding spaces.
\newblock In {\em Proceedings of the 2018 Conference on Empirical Methods in
  Natural Language Processing}, pages 1881--1890, 2018.

\bibitem{BackhoffBartlBeiglbockEder.20}
J.~Backhoff-Veraguas, D.~Bartl, M.~Beiglb{\"o}ck, and M.~Eder.
\newblock All adapted topologies are equal.
\newblock {\em Probab. Theory Related Fields}, 178(3-4):1125--1172, 2020.

\bibitem{BartlettLongLugosiTsigler.20}
P.~L. Bartlett, P.~M. Long, G.~Lugosi, and A.~Tsigler.
\newblock Benign overfitting in linear regression.
\newblock {\em Proc. Natl. Acad. Sci. USA}, 117(48):30063--30070, 2020.

\bibitem{BayraktarEcksteinZhang.22}
E.~Bayraktar, S.~Eckstein, and X.~Zhang.
\newblock Stability and sample complexity of divergence regularized optimal
  transport.
\newblock {\em Bernoulli}, 31(1):213--239, 2025.

\bibitem{BeiglbockHenryLaborderePenkner.11}
M.~Beiglb{\"o}ck, P.~Henry-Labord{\`e}re, and F.~Penkner.
\newblock Model-independent bounds for option prices: a mass transport
  approach.
\newblock {\em Finance Stoch.}, 17(3):477--501, 2013.

\bibitem{Belkin.21}
M.~Belkin.
\newblock Fit without fear: remarkable mathematical phenomena of deep learning
  through the prism of interpolation.
\newblock {\em Acta Numerica}, 30:203--248, 2021.

\bibitem{BerntonGhosalNutz.21}
E.~Bernton, P.~Ghosal, and M.~Nutz.
\newblock Entropic optimal transport: {G}eometry and large deviations.
\newblock {\em Duke Math. J.}, 171(16):3363--3400, 2022.

\bibitem{blondel18quadratic}
M.~Blondel, V.~Seguy, and A.~Rolet.
\newblock Smooth and sparse optimal transport.
\newblock volume~84 of {\em Proceedings of Machine Learning Research}, pages
  880--889, 2018.

\bibitem{Brezis.11}
H.~Brezis.
\newblock {\em Functional analysis, {S}obolev spaces and partial differential
  equations}.
\newblock Universitext. Springer, New York, 2011.

\bibitem{Brondsted.83}
A.~Br{\o}ndsted.
\newblock {\em An introduction to convex polytopes}, volume~90 of {\em Graduate
  Texts in Mathematics}.
\newblock Springer-Verlag, New York-Berlin, 1983.

\bibitem{Brualdi.06}
R.~A. Brualdi.
\newblock {\em Combinatorial matrix classes}, volume 108 of {\em Encyclopedia
  of Mathematics and its Applications}.
\newblock Cambridge University Press, Cambridge, 2006.

\bibitem{CarlierDuvalPeyreSchmitzer.17}
G.~Carlier, V.~Duval, G.~Peyr\'{e}, and B.~Schmitzer.
\newblock Convergence of entropic schemes for optimal transport and gradient
  flows.
\newblock {\em SIAM J. Math. Anal.}, 49(2):1385--1418, 2017.

\bibitem{CominettiSanMartin.94}
R.~Cominetti and J.~San~Mart\'{\i}n.
\newblock Asymptotic analysis of the exponential penalty trajectory in linear
  programming.
\newblock {\em Math. Programming}, 67(2, Ser. A):169--187, 1994.

\bibitem{ConfortiTamanini.19}
G.~Conforti and L.~Tamanini.
\newblock A formula for the time derivative of the entropic cost and
  applications.
\newblock {\em J. Funct. Anal.}, 280(11):108964, 2021.

\bibitem{Cuturi.13}
M.~Cuturi.
\newblock Sinkhorn distances: Lightspeed computation of optimal transport.
\newblock In {\em Advances in Neural Information Processing Systems 26}, pages
  2292--2300. 2013.

\bibitem{DesseinPapadakisRouas.18}
A.~Dessein, N.~Papadakis, and J.-L. Rouas.
\newblock Regularized optimal transport and the rot mover's distance.
\newblock {\em J. Mach. Learn. Res.}, 19(15):1--53, 2018.

\bibitem{DiMarinoGerolin.20b}
S.~Di~Marino and A.~Gerolin.
\newblock Optimal transport losses and {S}inkhorn algorithm with general convex
  regularization.
\newblock {\em Preprint arXiv:2007.00976v1}, 2020.

\bibitem{EcksteinKupper.21}
S.~Eckstein and M.~Kupper.
\newblock Computation of optimal transport and related hedging problems via
  penalization and neural networks.
\newblock {\em Appl. Math. Optim.}, 83(2):639--667, 2021.

\bibitem{EcksteinNutz.22}
S.~Eckstein and M.~Nutz.
\newblock Convergence rates for regularized optimal transport via quantization.
\newblock {\em Math. Oper. Res.}, 49(2):1223--1240, 2024.

\bibitem{EssidSolomon.18}
M.~Essid and J.~Solomon.
\newblock Quadratically regularized optimal transport on graphs.
\newblock {\em SIAM J. Sci. Comput.}, 40(4):A1961--A1986, 2018.

\bibitem{FinzelLi.00}
M.~Finzel and W.~Li.
\newblock Piecewise affine selections for piecewise polyhedral multifunctions
  and metric projections.
\newblock {\em J. Convex Anal.}, 7(1):73--94, 2000.

\bibitem{Galichon.16}
A.~Galichon.
\newblock {\em Optimal transport methods in economics}.
\newblock Princeton University Press, Princeton, NJ, 2016.

\bibitem{GeneveyEtAl.16}
A.~Genevay, M.~Cuturi, G.~Peyr\'{e}, and F.~Bach.
\newblock Stochastic optimization for large-scale optimal transport.
\newblock In {\em Advances in Neural Information Processing Systems 29}, pages
  3440--3448, 2016.

\bibitem{GonzalezSanzNutz.24b}
A.~Gonz{\'a}lez-Sanz and M.~Nutz.
\newblock Sparsity of quadratically regularized optimal transport: Scalar case.
\newblock {\em Preprint arXiv:2410.03353v1}, 2024.

\bibitem{GonzalezSanzNutzRiveros.24}
A.~Gonz{\'a}lez-Sanz, M.~Nutz, and A.~Riveros~Valdevenito.
\newblock Monotonicity in quadratically regularized linear programs.
\newblock {\em Preprint arXiv:2408.07871v1}, 2024.

\bibitem{GulrajaniAhmedArjovskyDumoulinCourville.17}
I.~Gulrajani, F.~Ahmed, M.~Arjovsky, V.~Dumoulin, and A.~Courville.
\newblock Improved training of {W}asserstein {GAN}s.
\newblock In {\em Proceedings of the 31st International Conference on Neural
  Information Processing Systems}, pages 5769--5779, 2017.

\bibitem{HagerZhang.16}
W.~W. Hager and H.~Zhang.
\newblock Projection onto a polyhedron that exploits sparsity.
\newblock {\em SIAM J. Optim.}, 26(3):1773--1798, 2016.

\bibitem{HallinDelBarrioCuestaAlbertosMatran.21}
M.~Hallin, E.~del Barrio, J.~Cuesta-Albertos, and C.~Matr\'an.
\newblock Distribution and quantile functions, ranks and signs in dimension
  {$d$}: a measure transportation approach.
\newblock {\em Ann. Statist.}, 49(2):1139--1165, 2021.

\bibitem{Hallin2010}
M.~Hallin, D.~Paindaveine, and M.~Siman.
\newblock Multivariate quantiles and multiple-output regression quantiles: from
  {$L_1$} optimization to halfspace depth.
\newblock {\em Ann. Statist.}, 38(2):635--669, 2010.

\bibitem{Koenker1978}
R.~Koenker and G.~Bassett, Jr.
\newblock Regression quantiles.
\newblock {\em Econometrica}, 46(1):33--50, 1978.

\bibitem{KolouriParlEtAl.17survey}
S.~Kolouri, S.~R. Park, M.~Thorpe, D.~Slepcev, and G.~K. Rohde.
\newblock Optimal mass transport: Signal processing and machine-learning
  applications.
\newblock {\em IEEE Signal Processing Magazine}, 34(4):43--59, 2017.

\bibitem{Leonard.12}
C.~L\'{e}onard.
\newblock From the {S}chr\"{o}dinger problem to the {M}onge-{K}antorovich
  problem.
\newblock {\em J. Funct. Anal.}, 262(4):1879--1920, 2012.

\bibitem{LiGenevayYurochkinSolomon.20}
L.~Li, A.~Genevay, M.~Yurochkin, and J.~Solomon.
\newblock Continuous regularized {W}asserstein barycenters.
\newblock In H.~Larochelle, M.~Ranzato, R.~Hadsell, M.~Balcan, and H.~Lin,
  editors, {\em Advances in Neural Information Processing Systems}, volume~33,
  pages 17755--17765. Curran Associates, Inc., 2020.

\bibitem{LiBayesian2010}
Q.~Li, R.~Xi, and N.~Lin.
\newblock Bayesian regularized quantile regression.
\newblock {\em Bayesian Anal.}, 5(3):533--556, 2010.

\bibitem{LorenzMahler.22}
D.~Lorenz and H.~Mahler.
\newblock Orlicz space regularization of continuous optimal transport problems.
\newblock {\em Appl. Math. Optim.}, 85(2):Paper No. 14, 33, 2022.

\bibitem{LorenzMannsMeyer.21}
D.~Lorenz, P.~Manns, and C.~Meyer.
\newblock Quadratically regularized optimal transport.
\newblock {\em Appl. Math. Optim.}, 83(3):1919--1949, 2021.

\bibitem{Mangasarian.84}
O.~L. Mangasarian.
\newblock Normal solutions of linear programs.
\newblock {\em Math. Programming Stud.}, 22:206--216, 1984.
\newblock Mathematical programming at Oberwolfach, II (Oberwolfach, 1983).

\bibitem{MangasarianMeyer.79}
O.~L. Mangasarian and R.~R. Meyer.
\newblock Nonlinear perturbation of linear programs.
\newblock {\em SIAM J. Control Optim.}, 17(6):745--752, 1979.

\bibitem{Mordant.23}
G.~Mordant.
\newblock Regularised optimal self-transport is approximate {G}aussian mixture
  maximum likelihood.
\newblock {\em Preprint arXiv:2310.14851v1}, 2023.

\bibitem{Nutz.24}
M.~Nutz.
\newblock Quadratically regularized optimal transport: Existence and
  multiplicity of potentials.
\newblock {\em Preprint arXiv:2404.06847v1}, 2024.

\bibitem{NutzWiesel.21}
M.~Nutz and J.~Wiesel.
\newblock Entropic optimal transport: convergence of potentials.
\newblock {\em Probab. Theory Related Fields}, 184(1-2):401--424, 2022.

\bibitem{Pal.19}
S.~Pal.
\newblock On the difference between entropic cost and the optimal transport
  cost.
\newblock {\em Ann. Appl. Probab.}, 34(1B):1003--1028, 2024.

\bibitem{PanaretosZemel.19}
V.~M. Panaretos and Y.~Zemel.
\newblock Statistical aspects of {W}asserstein distances.
\newblock {\em Annu. Rev. Stat. Appl.}, 6:405--431, 2019.

\bibitem{PeyreCuturi.19}
G.~Peyr{\'e} and M.~Cuturi.
\newblock Computational optimal transport: With applications to data science.
\newblock {\em Foundations and Trends in Machine Learning}, 11(5-6):355--607,
  2019.

\bibitem{RubnerTomasiGuibas.00}
Y.~Rubner, C.~Tomasi, and L.~J. Guibas.
\newblock The earth mover's distance as a metric for image retrieval.
\newblock {\em Int. J. Comput. Vis.}, 40:99--121, 2000.

\bibitem{seguy2018large}
V.~Seguy, B.~B. Damodaran, R.~Flamary, N.~Courty, A.~Rolet, and M.~Blondel.
\newblock Large scale optimal transport and mapping estimation.
\newblock In {\em International Conference on Learning Representations}, 2018.

\bibitem{Villani.03}
C.~Villani.
\newblock {\em Topics in optimal transportation}, volume~58 of {\em Graduate
  Studies in Mathematics}.
\newblock American Mathematical Society, Providence, RI, 2003.

\bibitem{Villani.09}
C.~Villani.
\newblock {\em Optimal transport, old and new}, volume 338 of {\em Grundlehren
  der Mathematischen Wissenschaften}.
\newblock Springer-Verlag, Berlin, 2009.

\bibitem{Weed.18}
J.~Weed.
\newblock An explicit analysis of the entropic penalty in linear programming.
\newblock volume~75 of {\em Proceedings of Machine Learning Research}, pages
  1841--1855, 2018.

\bibitem{WieselXu.24}
J.~Wiesel and X.~Xu.
\newblock Sparsity of quadratically regularized optimal transport: Bounds on
  concentration and bias.
\newblock {\em Preprint arXiv:2410.03425v1}, 2024.

\bibitem{ZhangMordantMatsumotoSchiebinger.23}
S.~Zhang, G.~Mordant, T.~Matsumoto, and G.~Schiebinger.
\newblock Manifold learning with sparse regularised optimal transport.
\newblock {\em Preprint arXiv:2307.09816v1}, 2023.

\bibitem{Zhou2024}
L.~Zhou, F.~Koehler, D.~J. Sutherland, and N.~Srebro.
\newblock Optimistic rates: A unifying theory for interpolation learning and
  regularization in linear regression.
\newblock {\em ACM / IMS Journal of Data Science}, 1(2):1--51, 2024.

\end{thebibliography}
\end{document}